\documentclass{article}
\usepackage{amsmath,amsfonts,amssymb,amsthm}
\usepackage{mathtools}
\usepackage{dsfont}
\usepackage{enumerate}
\usepackage{esint}
\allowdisplaybreaks[4]

\newtheorem{theorem}{Theorem}[section]
\newtheorem{proposition}[theorem]{Proposition}
\newtheorem{lemma}[theorem]{Lemma}

\newtheorem{remark}{Remark}[section]

\usepackage[colorlinks=true,linkcolor=blue,citecolor=green]{hyperref}

\usepackage[margin=2cm]{geometry}
\newlength{\bibitemsep}
\newlength{\bibparskip}
\let\oldthebibliography\thebibliography
\renewcommand\thebibliography[1]{%
  \oldthebibliography{#1}%
  \setlength{\parskip}{\bibitemsep}%
  \setlength{\itemsep}{\bibparskip}%
}

\numberwithin{equation}{section}

\begin{document}

\title{\Large\bf On the non-uniqueness of transport equation: the quantitative relationship between temporal and spatial regularity}
\author{Jingpeng Wu\thanks{School of Mathematics and Statistics; Institute of Artificial Intelligence, Huazhong University of Science and Technology, Wuhan, 430074, China (jpwu\_postdoc@hust.edu.cn).}}
\date{}
\maketitle

\begin{abstract}
In this paper, we consider the non-uniqueness of transport equation on the torus $\mathbb{T}^d$, with density $\rho\in L^{s}_tL_x^{p}$ and divergence-free vector field $\boldsymbol{u}\in L^{s'}_tL_x^{p'}\cap L^{\tilde{s}}_tW_x^{1,\tilde{p}}$. We prove that the non-uniqueness holds for $\frac{1}{p}+\frac{\tilde{s}'}{s\tilde{p}}>1+\frac{1}{d-1}$, with $d\ge 2$ and $s,p,\tilde{p}\in[1,\infty)$, $1\le\tilde{s}<s'$. The result can be extended to the transport-diffusion equation with diffusion operator of order $k$ in the class $\rho\in L^{s}_tL_x^{p}\cap L_t^{\bar{s}}C_x^{\bar{m}}$, $\boldsymbol{u}\in L^{s'}_tL_x^{p'}\cap L^{\tilde{s}}_tW_x^{1,\tilde{p}}$, under some conditions on $\bar{s},\bar{m},k$. In particular, when $\tilde{s}=1$, the additional condition is $\bar{m}<\frac{s}{\bar{s}}-1$, $k<\frac{s}{s'}+1$. These results can be considered as quantitative versions of Cheskidov and Luo's \cite{CL21,CL22}. The main tool is the convex integration developed by \cite{MS18,MS19,CL21,CL22}.

\noindent{\bf Keywords:} Transport equation; non-uniqueness; convex integration

\noindent{\bf MR Subject Classification:} 35A02; 35D30; 35Q35
\end{abstract}


\section{Introduction and main results}

This paper deals with the problem of (non-)uniqueness of weak solution to the linear transport equation on the torus $\mathbb{T}^d:=\mathbb{R}^d\setminus\mathbb{Z}^d$ written by
\begin{equation}\label{eq-TE}
\partial_t\rho+\boldsymbol{u}\cdot\nabla \rho=0,
\end{equation}
where $\rho\colon[0,T]\times\mathbb{T}^d\to\mathbb{R}$ is the unknown density, $\boldsymbol{u}\colon[0,T]\times\mathbb{T}^d\to\mathbb{R}^d$ is a given divergence-free vector field, i.e.~$\operatorname{div} \boldsymbol{u}=0$ in the sense of distribution. 

It is well known that for $ \boldsymbol{u}\in C_tW_x^{1,\infty}$, the well-posedness of \eqref{eq-TE} can be established by the method of characteristics. For non-Lipschitz vector fields, although the existence of weak solutions can be obtained by the method of regularization even for very rough vector fields, the uniqueness issue becomes subtler. The first breakthrough result traces back to the celebrated DiPerna-Lions theory \cite{DL89ODE} for Sobolev vector fields, then extended to BV vector fields by Ambrosio \cite{Amb04} using deep tools from geometric measure theory. Whereafter, there are abundant researches for the theory of non-smooth vector fields and their applications on non-linear PDEs, interested readers can refer to the surveys \cite{AC08,Amb17}. 

For the non-uniqueness issue, there have been examples at the Lagrangian level and at the Eulerian level, see details on the background and historical development in \cite{MS18,CL21}. In this paper, we mainly study the Eulerian construction of non-uniqueness, which was firstly obtained in \cite{CGSW15} for bounded vector fields, using the method of convex integration from \cite{DS09}. Later, a breakthrough result for the Sobolev vector field was obtained by Modena and Sz\'ekelyhidi \cite{MS18}. Based on their work, a series of works for the non-uniqueness to \eqref{eq-TE} have been done recently. When consider $\rho\in C_tL_x^{p}$, $\boldsymbol{u}\in C_tW_x^{1,q}$, the existence of non-unique weak solutions was established for $\frac{1}{p}+\frac{1}{q}>1+\frac{1}{d}$ in \cite{MS18,MS19,MS20}, and also in \cite{BCD21} with positive solutions. On the other hand, the DiPerna-Lions theory \cite{DL89ODE} provides the uniqueness for $\frac{1}{p}+\frac{1}{q}\le 1$. Moreover, the uniqueness holds for positive solutions with $\frac{1}{p}+\frac{1}{q}<1+\frac{1}{dp}$ \cite{BCD21}, and for $\frac{1}{p}+\frac{1}{q}<1+\frac{1}{d}$ with $p=1$ as well as the assumption that the so-called forward-backward integral curves are trivial \cite{CC21}. Hence it seems convincing that $\frac{1}{p}+\frac{1}{q}=1+\frac{1}{d}$ is the borderline situation for the uniqueness of the transport equation.

However, it was showed by Cheskidov and Luo \cite{CL21} that the existence of non-unique solutions can be achieved in the regime $\frac{1}{p}+\frac{1}{q}>1$ at the expense of temporal regularity. Further, the same authors \cite{CL22} provided the non-uniqueness for $\rho\in\cap_{s<\infty,k\in\mathbb{N}}L_t^sC_x^k$ and $u\in\cap_{q<\infty}L_t^1W_x^{1,q}$, which implies that the temporal integrability assumption in the uniqueness of the DiPerna-Lions theory is sharp. Such sharpness was also discovered in a revisited paper \cite{WZ23} of \cite{CL21}.

Hence, a natural question is whether there is a quantitative relationship between the spatial and temporal regularity for the non-uniqueness of weak solutions to \eqref{eq-TE}. More precisely, the question is to consider $\rho\in L^{s}_tL_x^{p}$, $\boldsymbol{u}\in L^{s'}_tL_x^{p'}\cap L^{\tilde{s}}_tW_x^{1,\tilde{p}}$, and search some function $f(s,p,\tilde{s},\tilde{p})$ and a constant $F_d$ such that the non-uniqueness holds for $f(s,p,\tilde{s},\tilde{p})>F_d$.

In this paper, we study the case $f=\frac{1}{p}+\frac{\tilde{s}'}{s\tilde{p}}$, $F_d=1+\frac{1}{d-1}$ and prove the following main theorem.
\begin{theorem}\label{thm1.1}
Let $d\ge 2$, $s,p,\tilde{p}\in[1,\infty)$, $1\le\tilde{s}<s'$ and
\begin{equation}\label{eq-assum}
\frac{1}{p}+\frac{\tilde{s}'}{s\tilde{p}}>1+\frac{1}{d-1}.
\end{equation}

Then there exists a divergence-free vector field
\begin{equation*}
\boldsymbol{u}\in L^{s'}(0,T;L^{p'}(\mathbb{T}^d))\cap L^{\tilde{s}}(0,T;W^{1,\tilde{p}}(\mathbb{T}^d))
\end{equation*}
such that the uniqueness of \eqref{eq-TE} fails in the class $\rho\in L^s(0,T;L^{p}(\mathbb{T}^d))$.
\end{theorem}
\begin{remark}\label{rem-1.1}
When $\tilde{s}=1$, $\tilde{s}'=\infty$, \eqref{eq-assum} holds for all $s,p,\tilde{p}<\infty$. The condition \eqref{eq-assum} is non-trivial for all $1\le\tilde{s}<s'$. For  $\tilde{s}=s'$, it degenerates into the Modena-Sz\'ekelyhidi condition $\frac{1}{p}+\frac{1}{\tilde{p}}>1+\frac{1}{d-1}$.
\end{remark}

\begin{remark}
It is interesting to consider the case $F_d=1+\frac{1}{d}$, i.e., the Modena-Sattig constant \cite{MS20}. It seems challenging to apply directly the method of \cite{MS20} to this paper, we are planning to investigate this issue in the future.
\end{remark}

\begin{remark}
A further challenging but also highly valuable problem is to uncover optimal $f(s,p,\tilde{s},\tilde{p})$ and $F_d$ such that the non-uniqueness holds for $f(s,p,\tilde{s},\tilde{p})>F_d$ and the uniqueness holds for $f(s,p,\tilde{s},\tilde{p})<F_d$.
\end{remark}

As in \cite{MS20,CL22}, Theorem~\ref{thm1.1} can be extended to the transport-diffusion equation with solutions with higher regularity.
\begin{theorem}\label{thm1.2}
Let $d\ge 2$, $s,p,\tilde{p}\in[1,\infty)$, $1\le\tilde{s}<s'$ and
\begin{align}\label{eq-assum-2}
&\frac{1}{p}+\frac{\tilde{s}'}{s\tilde{p}}>1+\frac{1}{d-1}.
\end{align}
Moreover, assume $\bar{m},k\in\mathbb{N}$, $1\le\bar{s}<s$ and
\begin{align}
\bar{m}<\frac{s}{\bar{s}}-1-(d-1)\left(\frac{s}{\bar{s}p'}-1\right)^{-}\text{ if }\tilde{s}>1,&\quad \bar{m}<\frac{s}{\bar{s}}-1\text{ if }\tilde{s}=1,\label{eq-assum-3}\\
k<\frac{s}{s'}+1-\frac{d-1}{p'}\left(\frac{s}{s'}-1\right)^{-}\text{ if }\tilde{s}>1,&\quad k<\frac{s}{s'}+1\text{ if }\tilde{s}=1.\label{eq-assum-4}
\end{align}

Then there exists a divergence-free vector field
\begin{equation*}
\boldsymbol{u}\in L^{s'}(0,T;L^{p'}(\mathbb{T}^d))\cap L^{\tilde{s}}(0,T;W^{1,\tilde{p}}(\mathbb{T}^d))
\end{equation*}
such that the uniqueness of the transport-diffusion equation  
\begin{equation}\label{eq-TDE}
\partial_t\rho+\boldsymbol{u}\cdot\nabla \rho+\mathcal{L}_k\rho=0
\end{equation}
fails in the class $\rho\in L^{s}(0,T;L^{p}(\mathbb{T}^d))\cap L^{\bar{s}}(0,T;C^{\bar{m}}(\mathbb{T}^d))$, where $\mathcal{L}_k$ is a given constant coefficient differential operator of order $k$.

\end{theorem}

\begin{remark}
Note for $\tilde{s}=1$ and $s\gg\max\{\bar{s},s'\}$, $\bar{m},k$ can be arbitrarily large, which is a quantitative description of Cheskidov and Luo's result \cite{CL22}. However, it is still not enough to answer the question (Q2) raised in \cite{BCC23}, i.e., whether non-uniqueness holds for advection-diffusion equation in the class $\boldsymbol{u}\in L_t^{2}L_x^{q}$, $\rho\in L_t^2H_x^{1}$ with $\frac{2d}{d+2}\le q<2$. Indeed, let $\mathcal{L}_k=-\Delta$, $\tilde{s}=\tilde{p}=\bar{m}=1$, $\bar{s}=2$, then Theorem~\ref{thm1.2} states that the non-uniqueness holds for \eqref{eq-TDE} in the class $\boldsymbol{u}\in L_t^{s'}L_x^{p'}\cap L_t^1W_x^{1,1}$, $\rho\in L_t^{s}L_x^{p}\cap L_t^2C_x^{1}$, for all $p\in[1,\infty)$, $s\in(4,\infty)$. In this case $s'<4/3$. Combing with Remark~\ref{rem-1.1}, our results seem to suggest that when both temporal integrability of $\boldsymbol{u}$ and $\rho$ are higher than a certain level, the time intermittency technique will be ineffective to obtain the non-uniqueness of weak solutions to \eqref{eq-TDE} with spatial regularity higher than the Modena-Sattig-Sz\'ekelyhidi setting.
\end{remark}

\begin{remark}
Similar results seem to hold for the following nonlinear transport equations
\begin{equation*}
\partial_t\rho+\boldsymbol{u}\cdot\nabla \rho+\mathcal{L}_k\rho+P(\rho)=0,
\end{equation*}
where $P\in C^{\infty}(\mathbb{R})$ satisfying $\vert P(x)-P(y)\vert\lesssim\vert x-y\vert^{\min(s,p)}$.
Indeed, we just need to modify the proof of Theorem~\ref{thm1.1} as in Sec.~\ref{sec-proof-of-theorem-2} by additionally estimating the nonlinear term
\begin{equation*}
\boldsymbol{R}_{\rm nonlinear}=\mathcal{R}[P(\rho+\theta)-P(\rho)].
\end{equation*}
\end{remark}


\subsection{Ideas of the proof}
For the case $p>1$, Theorem~\ref{thm1.1} can be proved by  the convex integration scheme of \cite{CL21,CL22}. The case $p=1$ is more delicate, since the method in \cite{MS19} can not be applied directly. We briefly describe two main ideas in this paper.
\begin{itemize}
\item  In \cite{MS19}, to deal with $\boldsymbol{R}_{\rm interact}$, there are two groups of parameters $(\mu',\sigma')$, $(\mu'',\sigma'')$ for spatial intermittency and oscillation satisfying $1\ll\mu'\ll\sigma'\ll\mu''\ll\sigma''$, of which the conditions always contradict \eqref{eq-assum}. Instead, we take advantage of the benefits of temporal intermittency and oscillation. More precisely, we require the partition of unity $\{\zeta_i\}$ used in \cite{MS19} satisfying a decay condition with respect to the temporal oscillation $\lambda$, see Sec.~\ref{subsec-inverse-flow}.

\item To deal with the defect field induced by the linear terms containing $\theta_p$, i.e., $\operatorname{div}\boldsymbol{R}=\theta_p$, where $\theta_p$ is of the form $a\operatorname{div}\boldsymbol{\Omega}\circ\Phi$ and $\Phi$ is the inverse flow map introduced in \cite{MS19}, we frequently use the following identity
\begin{equation*}
a\operatorname{div}\boldsymbol{\Omega}\circ\Phi=\operatorname{div}(a(\nabla\Phi)^{-1}\boldsymbol{\Omega}\circ\Phi)-\nabla a\cdot[(\nabla\Phi)^{-1}\boldsymbol{\Omega}\circ\Phi],
\end{equation*}
which implies that $\boldsymbol{R}$ can be constructed as
\begin{equation*}
\boldsymbol{R}:=a(\nabla\Phi)^{-1}\boldsymbol{\Omega}\circ\Phi-\mathcal{R}(\nabla a\cdot[(\nabla\Phi)^{-1}\boldsymbol{\Omega}\circ\Phi]).
\end{equation*}
See Sec.~\ref{subsec-new-solutions} for details.
\end{itemize}

\subsection{Notations}
\begin{itemize}
\item The norms of $L^r(\mathbb{T}^d)$, $L^s(0,T;L^p(\mathbb{T}^d))$ will be denoted standardly as $\|\cdot\|_{r}$, $\|\cdot\|_{L^s_tL^p_x}$ when there is no confusion. The norm of $C^k([0,T]\times\mathbb{T}^d)$ will be denoted as $\|\cdot\|_{C^k_{t,x}}$.

\item The spacial mean of $f\in L^1(\mathbb{T}^d)$ is $\fint_{\mathbb{T}^d}f\,\mathrm{d}x=\int_{\mathbb{T}^d}f\,\mathrm{d}x$. $C_0^{\infty}(\mathbb{T}^d)$ denotes the space of smooth periodic functions with zero mean. $C^{\infty}(\sigma^{-1}\mathbb{T}^d)$ denotes the space of smooth $\sigma^{-1}\mathbb{T}^d$-periodic functions.

\item $a\lesssim b$ will be denoted as $a\le Cb$ with some inessential constant $C$ might depending on the old solution $(\rho,\boldsymbol{u},\boldsymbol{R})$ and the constant $N$ but never on $\eta,\delta$ given in Proposition~\ref{prn3.1}. If the constant depends on some quantities, for instance $r$, we use the notation $a\lesssim_r b$.

\item We use the notation $a^{-}:=\max\{0,-a\}$ for $a\in\mathbb{R}$.
\end{itemize}

\paragraph{Overview of the paper.} The reslt of the paper is ogranized as follows.  In Sec.\ref{sec-main-proposition}, we prove Theorem~\ref{thm1.1} by applying Proposition~\ref{prn3.1}. In Sec.~\ref{sec-technical-tools}, we  collect the technical tools for convex integration. In Sec.~\ref{sec-proof-of-main-proposition}, we prove Proposition~\ref{prn3.1}. In Sec.~\ref{sec-proof-of-theorem-2}, we sketch the proof of Theorem~\ref{thm1.2}.

\section{Main Proposition and Proof of Theorem~\ref{thm1.1}}\label{sec-main-proposition}

Theorem~\ref{thm1.1} follows immediately from the following theorem.
\begin{theorem}\label{thm1.3}
Under the assumptions in Theorem~\ref{thm1.1}, for any $\epsilon>0$ and  any time-periodic $\tilde\rho\in C^{\infty}([0,T]\times\mathbb{T}^d)$ with constant mean
\begin{equation*}
\fint_{\mathbb{T}^d}\tilde\rho(t,x)\,\mathrm{d}x=\fint_{\mathbb{T}^d}\tilde\rho(0,x)\,\mathrm{d}x\text{ for all }t\in[0,T],
\end{equation*}
there exists a divergence-free vector field $\boldsymbol{u}$ and a density $\rho$ such that the following holds.
\begin{enumerate}[(i).]
\item $\boldsymbol{u}\in L^{s'}(0,T;L^{p'}(\mathbb{T}^d))\cap L^{\tilde{s}}(0,T;W^{1,\tilde{p}}(\mathbb{T}^d))$ and $\rho\in L^s(0,T;L^{p}(\mathbb{T}^d))$.
\item $\rho(t)$ is continuous in the distributional sense and for $t=0,T$, $\rho(t)=\tilde\rho(t)$.
\item $(\rho,\boldsymbol{u})$ is a weak solution to \eqref{eq-TE} with initial data $\tilde\rho(0)$.
\item The deviation of $L^p$ norm is small on average: $\|\rho-\tilde\rho\|_{L^s_tL^p_x}\le\epsilon$.
\end{enumerate}
\end{theorem}

\noindent\textbf{Proof of Theorem~\ref{thm1.1}.} Let $\bar{\rho}\in C_0^{\infty}(\mathbb{T}^d)$ with $\|\bar{\rho}\|_{L^p_x}=1$. We take $\tilde\rho=\chi(t)\bar{\rho}(x)$ with $\chi\in C^{\infty}([0,T];[0,1])$ satisfying $\chi=1$ if $\vert  t-\frac{T}{2}\vert \le \frac{T}{4}$ and $\chi=0$ if $\vert  t-\frac{T}{2}\vert  \ge\frac{3T}{8}$. We apply Theorem~\ref{thm1.3} with $\epsilon=\frac{1}{4}(\frac{T}{4})^{1/s}$ and obtain $(\rho,\boldsymbol{u})$ solving \eqref{eq-TE} with $\rho(0,x)\equiv0$. By the choice of $\epsilon$, we claim that $\rho$ cannot have a constant $L^p_x$ norm and obviously $\rho\not\equiv 0$, which implies the non-uniqueness.

Indeed, assume $\|\rho(t)\|_{L^p}\equiv C$ for some $C>0$. On the one hand, due to $\|\rho-\tilde\rho\|_{L^s_tL^p_x}\le \frac{1}{4}(\frac{T}{4})^{1/s}$, we have
\begin{equation*}\frac{T}{2}\vert C-1\vert \le\int_{\frac{T}{4}}^{\frac{3T}{4}}\|\rho(t)-\tilde\rho(t)\|_{L^p_x}\,\mathrm{d}t\le\big(\frac{T}{2}\big)^{1/s'}\|\rho-\tilde\rho\|_{L^s_tL^p_x}\le\frac{T}{8}, \end{equation*}
hence $\vert C-1\vert \le 1/4$, which implies $C>3/4$.

On the other hand,
\begin{equation*}
\frac{T}{8}\ge\big(\frac{T}{4}\big)^{1/s'}\|\rho-\tilde\rho\|_{L^s_tL^p_x}\ge\int_{[0,\frac{T}{8}]\cup[\frac{7T}{8},T]}\|\rho(t)-\tilde\rho(t)\|_{L^p_x}\,\mathrm{d}t=\frac{T}{4}C,
\end{equation*}
hence $C\le 1/2$, in contradiction with $C>3/4$, we prove the claim. \qed

To prove Theorem~\ref{thm1.3}, we follow the frameworks of \cite{MS18,CL21} to obtain space-time periodic approximate solutions $(\rho,\boldsymbol{u},\boldsymbol{R})$ to the transport equation by solving the continuity-defect equation
\begin{equation}\label{eq-CDE}
\left\{
\begin{split}
&\partial_t\rho+\operatorname{div}(\rho\boldsymbol{u})=\operatorname{div}\boldsymbol{R},\\
&\operatorname{div} \boldsymbol{u}=0,
\end{split}\right.
\end{equation}
where $\boldsymbol{R}\colon[0,1]\times\mathbb{T}^d\to\mathbb{R}^d$ is called the defect field. 

For any $0<r<1$, denote $I_r:=[r,1-r]$. To build a iteration scheme for proving Theorem~\ref{thm1.3}, we will construct the small perturbations on $I_r\times\mathbb{T}^d$ of $(\rho,\boldsymbol{u})$ to obtain a new solution $(\rho^1,\boldsymbol{u}^1,\boldsymbol{R}^1)$ such that the new defect field $\boldsymbol{R}^1$ has small $L^1_{t,x}$ norm. This is the following main proposition of the paper.

\begin{proposition}\label{prn3.1}
Under the assumptions in Theorem~\ref{thm1.1}, there exist a universal constant $M>0$ and a large integer $N\in\mathbb{N}$ such that the following holds.

Suppose $(\rho,\boldsymbol{u},\boldsymbol{R})$ is a smooth solution of \eqref{eq-CDE} on $[0,1]$. Then for any $\delta,\eta>0$, there exists another smooth solution $(\rho^1,\boldsymbol{u}^1,\boldsymbol{R}^1)$ of \eqref{eq-CDE} on $[0,1]$ which fulfills the estimates
\begin{align*}
\|\rho^1-\rho\|_{L^s_tL^p_x}&\le\eta M\|\boldsymbol{R}\|^{1/p}_{L^1_{t,x}},\\
\|\boldsymbol{u}^1-\boldsymbol{u}\|_{L^{s'}_tL^{p'}_x}&\le\eta^{-1} M\|\boldsymbol{R}\|^{1/p'}_{L^1_{t,x}},\\
\|\boldsymbol{u}^1-\boldsymbol{u}\|_{L^{\tilde s}_tW^{1,\tilde{p}}_x}&\le\delta,\quad\|\boldsymbol{R}^1\|_{L^1_{t,x}}\le\delta.
\end{align*}
In addition, the density perturbation $\rho^1-\rho$ has zero spacial mean and satisfies
\begin{align}
\Big\vert  \int_{\mathbb{T}^d}(\rho^1-\rho)(t,x)\phi(x)\,\mathrm{d}x\Big\vert  &\le\delta\|\phi\|_{C^N}\quad\forall t\in[0,1],\,\forall\phi\in C^{\infty}(\mathbb{T}^d),\label{eq-10}\\
\operatorname{supp}_t(\rho^1-\rho)&\in I_r\,\text{ for some }r>0.\label{eq-11}
\end{align}
\end{proposition}

Now we prove Theorem~\ref{thm1.3} by assuming Proposition~\ref{prn3.1}. We adopt the argument in \cite{MS20} which allows to deal with all $p\in[1,\infty)$.

\noindent\textbf{Proof of Theorem~\ref{thm1.3}.} Without loss of generality, we assume $T=1$. We will construct a sequence $(\rho^n,\boldsymbol{u}^n,\boldsymbol{R}^n)$ of solutions to \eqref{eq-CDE}. For $n=1$, we set
\begin{equation*}(\rho^1,\boldsymbol{u}^1,\boldsymbol{R}^1):=(\tilde\rho,0,\mathcal{R}(\partial_t\tilde\rho)),\end{equation*}
where $\mathcal{R}$ is the anti-divergence in the next section. Notice the constant mean assumption on $\tilde\rho$ implies zero mean of $\partial_t\tilde\rho$, hence $(\rho^1,\boldsymbol{u}^1,\boldsymbol{R}^1)$ solves \eqref{eq-CDE}.

Next we apply Proposition~\ref{prn3.1} inductively to obtain $(\rho^n,\boldsymbol{u}^n,\boldsymbol{R}^n)$ for $n=2,3\cdots$ as follows. Set $\eta_1:=\epsilon(2M\|\boldsymbol{R}^1\|^{1/p}_{L^1_{t,x}})^{-1}$ and choose sequence $\{(\delta_n,\eta_n)\}_{n=2}^{\infty}\subset(0,\infty)^2$ such that $\sum_{n}\delta_n^{1/2}=1$, $\delta_n^{1/p}\eta_n=\epsilon\delta^{1/2}_n/2M$. Observe that $\delta_n^{1/p'}/\eta_n=2M\delta^{1/2}_n/\epsilon$.

Given $(\rho^n,\boldsymbol{u}^n,\boldsymbol{R}^n)$, we apply Proposition~\ref{prn3.1} with parameters $\eta=\eta_n$ and $\delta=\delta_{n+1}$ to obtain a new triple $(\rho^{n+1},\boldsymbol{u}^{n+1},\boldsymbol{R}^{n+1})$ which verifies
\begin{align*}
\|\rho^{n+1}-\rho^n\|_{L^s_tL^p_x}&\le\eta_{n} M\|\boldsymbol{R}^n\|^{1/p}_{L^1_{t,x}},\\
\|\boldsymbol{u}^{n+1}-\boldsymbol{u}^n\|_{L^{s'}_tL^{p'}_x}&\le\eta_{n}^{-1} M\|\boldsymbol{R}^n\|^{1/p'}_{L^1_{t,x}},\\
\|\boldsymbol{u}^{n+1}-\boldsymbol{u}^n\|_{L^{\tilde s}_tW^{1,\tilde{p}}_x}&\le\delta_{n+1},\quad\|\boldsymbol{R}^{n+1}\|_{L^1_{t,x}}\le\delta_{n+1},\\
\Big\vert \int_{\mathbb{T}^d}(\rho^{n+1}-\rho^n)(t,x)\phi(x)\,\mathrm{d}x\Big\vert &\le\delta_{n+1}\|\phi\|_{C^N}\quad\forall t\in[0,1],\,\forall\phi\in C^{\infty}(\mathbb{T}^d).
\end{align*}
When $n\ge 2$, we have
\begin{align*}
\|\rho^{n+1}-\rho^n\|_{L^s_tL^p_x}&\le\frac{\epsilon\delta^{1/2}_n}{2},\\
\|\boldsymbol{u}^{n+1}-\boldsymbol{u}^n\|_{L^{s'}_tL^{p'}_x}&\le\frac{2M^2\delta^{1/2}_n}{\epsilon}.
\end{align*}
Clearly there are functions $\rho\in L^s_tL^p_x$ and $\boldsymbol{u}\in L^{s'}_tL^{p'}_x\cap L^{\tilde{s}}_tW^{1,\tilde{p}}_x$ such that $\rho^{n}\to\rho$ in $L^s_tL^p_x$ and $\boldsymbol{u}^n\to\boldsymbol{u}$ in $L^{s'}_tL^{p'}_x\cap L^{\tilde{s}}_tW^{1,\tilde{p}}_x$. Moreover, we have $\rho^{n}\boldsymbol{u}^n\to\rho\boldsymbol{u}$ and $\boldsymbol{R}^n\to 0$ in $L^1_{t,x}$, and  $\int_{\mathbb{T}^d}\rho^n(\cdot,x)\phi(x)\,\mathrm{d}x\to\int_{\mathbb{T}^d}\rho(\cdot,x)\phi(x)\,\mathrm{d}x$ in $L^{\infty}_t$. Combine the fact $\operatorname{supp}_t(\rho^{n+1}-\rho^n)\in I_{r_n}$ for some $r_n>0$, we obtain the temporal continuity of $\rho$ in the distributional sense and for $t=0,1$, $\rho(t)=\tilde\rho(t)$, furthermore $(\rho,\boldsymbol{u})$ is a weak solution to \eqref{eq-TE} with initial data $\tilde\rho(0)$.

Finally, thanks to the choice of $\{\delta_n\},\{\eta_n\}$, we have
\begin{equation*}
\|\rho-\tilde\rho\|_{L^s_tL^p_x}\le\|\rho^{2}-\rho^1\|_{L^s_tL^p_x}+\sum_{n=2}^{\infty}\|\rho^{n+1}-\rho^n\|_{L^s_tL^p_x}\le\frac{\epsilon}{2}+\frac{\epsilon}{2}\sum_{n=2}^{\infty}\delta^{1/2}_n=\epsilon.
\end{equation*}\qed

%
%


\section{Technical tools}\label{sec-technical-tools}

In this section, we collect the technical tools well prepared in \cite{MS19,CL21,CL22}. To keep it concise, many thought-provoking details have been condensed, interested readers can refer to \cite{MS18,MS19,MS20,CL21,CL22}.

\subsection{Diffeomorphisms of the flat torus}\label{subsec-inverse-flow}

To handle the case $p=1$, we employ the inverse flow map associated to $\boldsymbol{u}$ introduced in \cite{MS19}, which originates from the framework of the Euler equation in \cite{BDIS15}.

We say that $\Phi\colon\mathbb{T}^d\to\mathbb{T}^d$ is a diffeomorphism of $\mathbb{T}^d$, if $\Phi(x+k)=\Phi(x)+k$ for all $k\in\mathbb{Z}^d$. The diffeomorphism $\Phi$ is measure-preserving if $\vert\operatorname{det}\nabla\Phi(x)\vert=1$ for all $x\in\mathbb{T}^d$. Note $\nabla\Phi(x)\in\mathbb{R}^{d\times d}$ is an invertible matrix, its inverse matrix denotes as $(\nabla\Phi)^{-1}$. For every $m\in\mathbb{N}$, we have
\begin{align}
&\|\nabla^k((\nabla\Phi)^{-1})\|_{C^0(\mathbb{T}^d)}\lesssim_{k}\|\nabla\Phi\|_{C^k(\mathbb{T}^d)}^{d-1},\\
&\|g\circ\Phi\|_{p}=\|g\|_{p},\quad\|g\circ\Phi\|_{W^{k,p}}\lesssim_{k}\|\nabla\Phi\|_{C^{k-1}}^k\|g\|_{W^{k,p}}.\label{eq-3.2}
\end{align}
Let $G\in C^{\infty}(\mathbb{T}^d;\mathbb{R}^d)$, then
\begin{equation}\label{eq-div-inverse-map}
\operatorname{div}[(\nabla\Phi)^{-1}G\circ\Phi]=(\operatorname{div}G)\circ\Phi.
\end{equation}

Now we define the measure-preserving diffeomorphism $\Phi_i$ associated to $\boldsymbol{u}$ in small intervals such that $(\nabla\Phi_i)^{-1}$ is close to the identity matrix $\operatorname{Id}$. Set $D:=1/\nu\in\mathbb{N}$, for every $i=1\dots N$, let $I_i:=[i\nu,(i+1)\nu]$ and let $t_i:=(i+1/2)\nu$. Consider  a partition of unity $\{\zeta_i\}$ subordinate to the family of intervals $\{I_i\}$. For every $i=1\dots N$, let $\Phi_i\colon[0,1]\times\mathbb{T}^d\to\mathbb{T}^d$ be the solution to
\begin{equation*}
\partial_t\Phi_i+\boldsymbol{u}\cdot\nabla \Phi_i=0,\quad\Phi_i\vert_{t=t_i}=x.
\end{equation*}
Since $(\nabla\Phi_i(t_i))^{-1}=\operatorname{Id}$, when $\nu$ is small enough, we have
\begin{equation}\label{eq-3.5}
\|(\nabla\Phi_i(t_i))^{-1}-\operatorname{Id}\|_{C^0([i\nu,(i+1)\nu]\times\mathbb{T}^d)}=O(\nu).
\end{equation}
Moreover, we let
\begin{equation}
\operatorname{meas}\{\operatorname{supp}(\zeta_i\zeta_{i+1})\}=O(\lambda^{-1/2}),\quad\|\dot{\zeta}_i\|_{C^0([0,1])}\lesssim\lambda^{1/2},\label{eq-3.6}
\end{equation}
where $\lambda$ is the temporal oscillation parameter determined later.

\subsection{Anti-divergence operators and common inequalities}

It is well known that for any $f\in C^{\infty}(\mathbb{T}^d)$, there exists a unique solution in $C_0^{\infty}(\mathbb{T}^d)$ of the Poisson equation
\begin{equation*}
\Delta u=f-\fint_{\mathbb{T}^d} f\,\mathrm{d}x.
\end{equation*}
Hence the standard anti-divergence operator $\mathcal{R}\colon C^{\infty}(\mathbb{T}^d)\to C^{\infty}_0(\mathbb{T}^d;\mathbb{R}^d)$ can be well defined as
\begin{equation*}
\mathcal{R}f:=\Delta^{-1}\nabla  f,
\end{equation*}
which satisfies
\begin{equation}
\operatorname{div}(\mathcal{R}f)=f-\fint_{\mathbb{T}^d} f\,\mathrm{d}x.
\end{equation}
For every $m\in\mathbb{N}$, $r\in[1,\infty]$, $\mathcal{R}$ is bounded on Sobolev spaces $W^{m,r}(\mathbb{T}^d)$: 
\begin{equation}\label{eq-Wkp-bound-anti}
\|\mathcal{R}f\|_{W^{m,r}}\lesssim_{m,r}\|f\|_{W^{m,r}}.
\end{equation}

In \cite{MS19}, the authors introduced an improved anti-divergence operator involving a diffeomorphism. Let $f\in C^{\infty}(\mathbb{T}^d)$, $g\in C_0^{\infty}(\mathbb{T}^d)$, $\sigma\in\mathbb{N}$ and $\Phi\in C^{\infty}(\mathbb{T}^d;\mathbb{T}^d)$ be a measure-preserving diffeomorphism. Then there exists a vector field $u\in C^{\infty}(\mathbb{T}^d)$ written as $u:=\mathcal{R}[f(\cdot)g(\sigma\Phi(\cdot))]$ such that
\begin{equation*}
\operatorname{div}u=f(\cdot)g(\sigma\Phi(\cdot))-\fint_{\mathbb{T}^d}  f(x)g(\sigma\Phi(x))\,\mathrm{d}x
\end{equation*}
and for every $m\in\mathbb{N}$, $r\in[1,\infty]$, there holds (see \cite{MS19} for the proof)
\begin{equation}\label{eq-Wkp-bound-anti-with-inverse-map}
\|\mathcal{R}[f(\cdot)g(\sigma\Phi(\cdot))]\|_{W^{m,r}(\mathbb{T}^d)}\lesssim_{m,r}\sigma^{m-1}\|f\|_{C^{m+1}(\mathbb{T}^d)}\|\nabla\Phi\|_{C^{m}(\mathbb{T}^d)}^{d-1+m}\|g\|_{W^{m,r}(\mathbb{T}^d)}.
\end{equation}


The following two inequalities will be frequently used, the proofs can be referred to \cite{MS19}. Assume $d\ge2$, $m,\sigma\in\mathbb{N}$, $r\in[1,\infty]$, $q\in(1,\infty)$, $n\ge 0$ is an even integer, $a,f\in C^{\infty}(\mathbb{T}^d)$, $g\in C_0^{\infty}(\mathbb{T}^d)$and $\Phi\in C^{\infty}(\mathbb{T}^d;\mathbb{T}^d)$ be a measure-preserving diffeomorphism, then there hold
\begin{align}
\text{Riemann-Lebesgue type inequality}\qquad&\Big\vert \fint_{\mathbb{T}^d}a(x)g(\sigma\Phi(x))\,\mathrm{d}x\Big\vert  \lesssim_n \sigma^{-n}\|a\|_{C^n}\|\nabla\Phi\|_{C^{n-1}}^{d-1}\|g\|_{2},\label{eq-Riemann-Lebesgue}\\
\text{Improved H\"older's inequality}\qquad&\|a(\cdot)f(\sigma\Phi(\cdot))\|_{r}\lesssim_{r}\Big(\|a\|_{r}+\sigma^{-\frac{1}{r}}\|a\|_{C^1}\|\nabla\Phi\|_{C^{0}}^{d-1}\Big)\|f\|_{r}.\label{eq-Improved-Holder}
\end{align}
In particular, if $\Phi=x$, we have the classical forms of these inequalities
\begin{align}
\text{Riemann-Lebesgue type inequality}\qquad&\Big\vert \fint_{\mathbb{T}^d}a(x)g(\sigma x)\,\mathrm{d}x\Big\vert  \lesssim_n \sigma^{-n}\|a\|_{C^n}\|g\|_{2},\label{eq-Riemann-Lebesgue-classical}\\
\text{Improved H\"older's inequality}\qquad&\|a(\cdot)f(\sigma\cdot)\|_{r}\lesssim_{r}\Big(\|a\|_{r}+\sigma^{-\frac{1}{r}}\|a\|_{C^1}\Big)\|f\|_{r}.\label{eq-Improved-Holder-classical}
\end{align}

\subsection{Building blocks}

We recall the building blocks for the convex integration construction introduced in \cite{CL21,CL22}, which originates from \cite{DS17,MS18,BV19b}. 

The spatial part consists in the Mikado densities $\Theta_j^{\mu,\sigma}$, fields $\boldsymbol{W}_j^{\mu,\sigma}=W_j^{\mu,\sigma}\boldsymbol{e}_j$ and potentials $\boldsymbol{\Omega}_j^{\mu,\sigma}$, $j=1\dots d$ satisfying: for all $\mu\ge 1$, $\sigma,m\in\mathbb{N}$, $r\in[1,\infty]$, then $\Theta_j^{\mu,\sigma},W_j^{\mu,\sigma}\in C_0^{\infty}(\sigma^{-1}\mathbb{T}^d)$, $\boldsymbol{\Omega}_j^{\mu,\sigma}\in C^{\infty}(\sigma^{-1}\mathbb{T}^d;\mathbb{R}^d)$ and
\begin{align}
&\operatorname{div}(\Theta_j^{\mu,\sigma}\boldsymbol{W}_j^{\mu,\sigma})=0,\quad\operatorname{div} \boldsymbol{W}_j^{\mu,\sigma}=0,\quad\operatorname{div}\boldsymbol{\Omega}_j^{\mu,\sigma}=\sigma\Theta_j^{\mu,\sigma},\label{eq-7}\\
&\int_{\mathbb{T}^d}\Theta_j^{\mu,\sigma}(x)W_j^{\mu,\sigma}(x)\,\mathrm{d}x=1\text{ for all } 1\le j\le d,\label{eq-8}\\
&\|\nabla ^m\Theta_j^{\mu,\sigma}\|_{r}\lesssim_{m,r}\sigma^m\mu^{m+\frac{d-1}{p}-\frac{d-1}{r}},\label{eq-Mikado-density-estimate}\\
&\|\nabla ^m W_j^{\mu,\sigma}\|_{r}\lesssim_{m,r}\sigma^m\mu^{m+\frac{d-1}{p'}-\frac{d-1}{r}},\label{eq-Mikado-field-estimate}\\
&\|\nabla ^m\boldsymbol{\Omega}_j^{\mu,\sigma}\|_{r}\lesssim_{m,r}\sigma^m\mu^{m-1+\frac{d-1}{p}-\frac{d-1}{r}}.\label{eq-Mikado-potential-estimate}
\end{align}

The temporal part consists in the temporal functions $\bar{g}_j^{\kappa,\lambda},\tilde{g}_j^{\kappa,\lambda}$ and temporal correction function $h_j^{\kappa,\lambda}$ satisfying: for all $\kappa\ge 1$, $\lambda,m\in\mathbb{N}$, $r\in[1,\infty]$, then $\bar{g}_j^{\kappa,\lambda},\tilde{g}_j^{\kappa,\lambda}\in C_0^{\infty}(\lambda^{-1}\mathbb{T})$, $h_j^{\kappa,\lambda}\in C^{\infty}(\lambda^{-1}\mathbb{T})$ and
\begin{align}
&\int_0^1\bar{g}_j^{\kappa,\lambda}(t)\tilde{g}_j^{\kappa,\lambda}(t)\,\mathrm{d}t=1,\quad h_j^{\kappa,\lambda}(t):=\lambda\int_0^t(\bar{g}_j^{\kappa,\lambda}\tilde{g}_j^{\kappa,\lambda}-1)\,\mathrm{d}\tau,\label{eq-temporal-equality}\\
&\|\partial_t^m\bar{g}_j^{\kappa,\lambda}\|_{r}\lesssim \lambda^m\kappa^{m+\frac{1}{s'}-\frac{1}{r}},\quad\|\partial_t^m\tilde{g}_j^{\kappa,\lambda}\|_{r}\lesssim \lambda^m\kappa^{m+\frac{1}{s}-\frac{1}{r}},\quad\|h_j^{\kappa,\lambda}\|_{\infty}\le 1.\label{eq-temporal-intermittency-estimate}
\end{align}

\subsection{Smooth cutoff of the defect fields}\label{subsec-smooth-cutoff}

We write $\boldsymbol{R}=\sum_j R_j\boldsymbol{e}_j$, where $\boldsymbol{e}_j$ is the $j$-th standard Euclidean basis.

Recall the notation $I_r=[r,1-r]\subset(0,1)$ for $0<r<1$. Define the smooth cutoff functions $\chi_{j}\in C_c^{\infty}(\mathbb{R}\times\mathbb{T}^d)$ satisfying
\begin{equation}\label{eq-14}
0\le\chi_{j}\le1,\quad\chi_{j}(t,x)=\left\{
\begin{split}
0,&\text{ if }\vert R_j\vert \le \frac{\delta}{16d}\text{ or }t\notin I_{r/2},\\
1,&\text{ if }\vert R_j\vert \ge \frac{\delta}{8d}\text{ and }t\in I_{r}.
\end{split}\right.
\end{equation}
Where $r>0$ is fixed sufficiently small enough such that 
\begin{equation}\label{eq-15}
\|\boldsymbol{R}\|_{L^{\infty}_{t,x}}\le\frac{\delta}{16rd}.
\end{equation}
Notice $\operatorname{supp}\chi_{j}\subset I_{r/2}\subset(0,1)$. By a slight abuse of notation, $\chi_{j}$ denote the 1-periodic extension in time of $\chi_{j}$. Define $\widetilde{R}_j:=\chi_{j}R_j$ and
\begin{equation*}
a_j(t,x):=\Big(\frac{\|\widetilde{R}_j(t)\|_{L^1}}{\|\widetilde{R}_j\|_{L^{1}_{t,x}}}\Big)^{\frac{1}{s}-\frac{1}{p}}\operatorname{sign}(-R_j)\chi_j\vert R_j\vert ^{\frac{1}{p}},\quad b_j(t,x):=\Big(\frac{\|\widetilde{R}_j(t)\|_{L^1}}{\|\widetilde{R}_j\|_{L^{1}_{t,x}}}\Big)^{\frac{1}{p}-\frac{1}{s}}\chi_j\vert R_j\vert ^{\frac{1}{p'}}.
\end{equation*}
$a_j,b_j$ have the following properties proved in \cite[Lem.7.1]{CL21}. 
\begin{align}
&a_jb_j=-\chi_j^2R_j,\quad a_j,b_j\in C^{\infty}([0,1]\times\mathbb{T}^d),\quad \|\widetilde{R}_j(t)\|_{L^1}\in C^{\infty}([0,1]),\label{eq-ajbj-equality}\\
&\|a_j(t)\|_{L^p_x}\le\|\widetilde{R}_j\|_{L^{1}_{t,x}}^{\frac{1}{p}-\frac{1}{s}}\|\widetilde{R}_j(t)\|_{L^1_x}^{\frac{1}{s}},\quad\|b_j(t)\|_{L^{p'}}\le\|\widetilde{R}_j\|_{L^{1}_{t,x}}^{\frac{1}{s}-\frac{1}{p}}\|\widetilde{R}_j(t)\|_{L^1_x}^{\frac{1}{s'}},\label{eq-Lp-bound-ajbj}\\
&\|a_j\|_{C^k_{t,x}}\le C,\quad\|b_j\|_{C^k_{t,x}}\le C.\label{eq-Ck-bound-ajbj}
\end{align}

\section{Proof of Proposition~\ref{prn3.1}}\label{sec-proof-of-main-proposition}

\subsection{Construction of new solutions}\label{subsec-new-solutions}

Now we define the perturbation of density as  $\theta=\theta_p+\theta_c+\theta_o$ and the perturbation of vector field as $\boldsymbol{w}=\boldsymbol{w}_p+\boldsymbol{w}_c$ where
\begin{align*}
\theta_p(t,x)&:=\sum\nolimits_{i,j}\eta \tilde{g}_j^{\kappa,\lambda}(t)a_j(t,x)\zeta_i(t)\Theta_j^{\mu,\sigma}\circ\Phi_{i}\\
\boldsymbol{w}_p(t,x)&:=\sum\nolimits_{i,j}\eta^{-1}\bar{g}_j^{\kappa,\lambda}(t)b_j(t,x)\zeta_i(t)[(\nabla \Phi_{i})^{-1}\boldsymbol{e}_j]W_j^{\mu,\sigma}\circ\Phi_{i}\\
\theta_c(t)&:=-\fint_{\mathbb{T}^d}\theta_p(t,x)\,\mathrm{d}x,\\
\boldsymbol{w}_c(t,x)&:=-\sum\nolimits_{i,j}\eta^{-1}\bar{g}_j^{\kappa,\lambda}\zeta_{i}\mathcal{R}\{\nabla b_j\cdot[(\nabla \Phi_{i})^{-1}\boldsymbol{e}_j]W_j^{\mu,\sigma}\circ\Phi_{i}\},\\
\theta_o(t,x)&:=\sum\nolimits_{j}\lambda^{-1} h_j^{\kappa,\lambda}(t)\partial_j(\chi_j^2R_j).
\end{align*}
It is easy to check that $\theta$ has zero mean and $\boldsymbol{w}$ is divergence free. Then $\rho^1,\boldsymbol{u}^1$ can be constructed as
\begin{align*}
&\rho^1:=\rho+\theta,\quad \boldsymbol{u}^1:=\boldsymbol{u}+\boldsymbol{w}.
\end{align*}

The next step is to construct $\boldsymbol{R}^1$ such that $(\rho^1,\boldsymbol{u}^1,\boldsymbol{R}^1)$ satisfying the continuity-defect equation
\begin{equation*}
\partial_t\rho^1+\operatorname{div}(\rho^1\boldsymbol{u}^1)=\operatorname{div} \boldsymbol{R}^1.
\end{equation*}
Insert $\partial_t\rho+\operatorname{div}(\rho\boldsymbol{u})=\operatorname{div} \boldsymbol{R}$ into this new continuity-defect equation, we obtain
\begin{equation*}
\partial_t\theta+\operatorname{div}(\rho\boldsymbol{w}+\theta\boldsymbol{u}+\theta\boldsymbol{w}+\boldsymbol{R})=\operatorname{div} \boldsymbol{R}^1.
\end{equation*}
Then we split $\boldsymbol{R}^1$ as follows:
\begin{align*}
\boldsymbol{R}^1:=\boldsymbol{R}_{\rm lin}+\boldsymbol{R}_{\rm cor}+\boldsymbol{R}_{\rm trans}+\boldsymbol{R}_{\rm osc},
\end{align*}
where
\begin{align*}
\boldsymbol{R}_{\rm lin}&:=\theta_0\boldsymbol{u}+\rho \boldsymbol{w},\quad\boldsymbol{R}_{\rm cor}:=\theta \boldsymbol{w}_c+(\theta_o+\theta_c)\boldsymbol{w}_p,\\
\boldsymbol{R}_{\rm trans}&:=\boldsymbol{R}_{\rm trans,1}+\boldsymbol{R}_{\rm trans,2},\\
\boldsymbol{R}_{\rm trans,1}&:=\sigma^{-1}\eta\sum\nolimits_{i,j}\Big\{\partial_t(\tilde{g}_j^{\kappa,\lambda}a_j)\zeta_{i}(\nabla \Phi_{i})^{-1}\boldsymbol{\Omega}_j^{\mu,\sigma}\circ\Phi_{i}\\
&\qquad\qquad-\mathcal{R}[\partial_t(\tilde{g}_j^{\kappa,\lambda}\nabla a_j)\cdot(\zeta_{i}(\nabla \Phi_{i})^{-1}\boldsymbol{\Omega}_j^{\mu,\sigma}\circ\Phi_{i})]\Big\},\\
\boldsymbol{R}_{\rm trans,2}&:=\eta\sum\nolimits_{i,j}\mathcal{R}\Big\{\tilde{g}_j^{\kappa,\lambda}[a_j\dot{\zeta}_{i}+(\boldsymbol{u}\cdot\nabla a_j)\zeta_{i}]\Theta_j^{\mu,\sigma}\circ\Phi_{i}\Big\},\\
\boldsymbol{R}_{{\rm osc}}&:=\boldsymbol{R}_{{\rm osc},x}+\boldsymbol{R}_{{\rm osc},t}+\boldsymbol{R}_{{\rm rem}}+\boldsymbol{R}_{\rm flow}+\boldsymbol{R}_{\rm interact},\\
\boldsymbol{R}_{{\rm osc},x}&:=-\sum\nolimits_{i,j}\tilde{g}_j^{\kappa,\lambda}\bar{g}_j^{\kappa,\lambda}\zeta_{i}^2\mathcal{R}\Big\{\nabla(R_j\chi_j^2)\cdot[(\nabla \Phi_{i})^{-1}\boldsymbol{e}_j][(\Theta_j^{\mu,\sigma}W_j^{\mu,\sigma})\circ\Phi_{i}-1]\Big\},\\
\boldsymbol{R}_{{\rm osc},t}&:=\lambda^{-1}\sum\nolimits_{j}h_j^{\kappa,\lambda}(t)\partial_t(\chi_j^2R_j)\boldsymbol{e}_j,\quad\boldsymbol{R}_{\rm rem}:=\sum\nolimits_{j}(1-\chi_j^2)R_j\boldsymbol{e}_j,\\
\boldsymbol{R}_{\rm flow}&:=\sum\nolimits_{i,j}\tilde{g}_j^{\kappa,\lambda}\bar{g}_j^{\kappa,\lambda}\chi_j^2\zeta_{i}^2[\operatorname{Id}-(\nabla\Phi_{i})^{-1}]R_j\boldsymbol{e}_j,\\
\boldsymbol{R}_{\rm interact}&:=-\sum\nolimits_{j}\tilde{g}_j^{\kappa,\lambda}\bar{g}_j^{\kappa,\lambda}R_j\chi_j^2\sum_{i=1}^{D-1}\zeta_{i}\zeta_{i+1}\Big\{[(\nabla \Phi_{i})^{-1}\boldsymbol{e}_j](W_j^{\mu,\sigma}\circ\Phi_{i})(\Theta_j^{\mu,\sigma}\circ\Phi_{i+1})\\
&\qquad\qquad\qquad\qquad\qquad\qquad+[(\nabla \Phi_{i+1})^{-1}\boldsymbol{e}_j](W_j^{\mu,\sigma}\circ\Phi_{i+1})(\Theta_j^{\mu,\sigma}\circ\Phi_{i})\Big\}.
\end{align*}
It can be proved by the properties of the building blocks (\ref{eq-7},\ref{eq-8},\ref{eq-temporal-equality},\ref{eq-ajbj-equality}) that
\begin{align*}
&\operatorname{div}\boldsymbol{R}_{\rm lin}=\operatorname{div}(\theta_0\boldsymbol{u}+\rho\boldsymbol{w}),\quad\operatorname{div}\boldsymbol{R}_{\rm cor}=\operatorname{div}(\theta\boldsymbol{w}_c+\theta_{o}\boldsymbol{w}_p+\theta_c\boldsymbol{w}_p),\\
&\operatorname{div}\boldsymbol{R}_{\rm trans}=\partial_t(\theta_p+\theta_c)+\operatorname{div}(\theta_p\boldsymbol{u}+\theta_c\boldsymbol{u}),\quad\operatorname{div}\boldsymbol{R}_{\rm osc}=\partial_t\theta_{o}+\operatorname{div}(\theta_p\boldsymbol{w}_p+\boldsymbol{R}).
\end{align*}


Indeed, the construction of $\boldsymbol{R}_{\rm trans}$ is based on the following facts: recall \eqref{eq-div-inverse-map} and \eqref{eq-7}, we have $\operatorname{div}[(\nabla\Phi_i)^{-1}\boldsymbol{\Omega}_j^{\mu,\sigma}\circ\Phi_i]=\sigma\Theta_j^{\mu,\sigma}\circ\Phi_i$, hence
\begin{align*}
&\partial_t\theta_p+\operatorname{div}(\theta_p\boldsymbol{u}+\theta_c\boldsymbol{u})\\
&=\partial_t\theta_p+\boldsymbol{u}\cdot\nabla \theta_p\\
&=\eta\sum\nolimits_{i,j}\left\{\partial_t(\tilde{g}_j^{\kappa,\lambda}a_j)\zeta_{i}+\tilde{g}_j^{\kappa,\lambda}a_j\dot{\zeta}_{i}+\tilde{g}_j^{\kappa,\lambda}(\boldsymbol{u}\cdot\nabla a_j)\zeta_{i}\right\}\Theta_j^{\mu,\sigma}\circ\Phi_{i}\\
&\quad+\eta\sum\nolimits_{i,j}\tilde{g}_j^{\kappa,\lambda}a_j\zeta_{i}((\nabla \Theta_j^{\mu,\sigma})\circ\Phi_{i})\cdot(\partial_t\Phi_{i}+\boldsymbol{u}\cdot\nabla \Phi_{i})\\
&=\sigma^{-1}\eta\sum\nolimits_{i,j}\operatorname{div}\{\partial_t(\tilde{g}_j^{\kappa,\lambda}a_j)\zeta_{i}(\nabla \Phi_{i})^{-1}\boldsymbol{\Omega}_j^{\mu,\sigma}\circ\Phi_{i}\}\\
&\quad-\sigma^{-1}\eta\sum\nolimits_{i,j}\partial_t(\tilde{g}_j^{\kappa,\lambda}\nabla a_j)\cdot(\zeta_{i}(\nabla \Phi_{i})^{-1}\boldsymbol{\Omega}_j^{\mu,\sigma}\circ\Phi_{i})\\
&\quad+\eta\sum\nolimits_{i,j}\tilde{g}_j^{\kappa,\lambda}\left\{a_j\dot{\zeta}_{i}+(\boldsymbol{u}\cdot\nabla a_j)\zeta_{i}\right\}\Theta_j^{\mu,\sigma}\circ\Phi_{i},
\end{align*}
and
\begin{align*}
&\int_{\mathbb{T}^d}-\sigma^{-1}\eta\sum\nolimits_{i,j}\partial_t(\tilde{g}_j^{\kappa,\lambda}\nabla a_j)\cdot(\zeta_{i}(\nabla \Phi_{i})^{-1}\boldsymbol{\Omega}_j^{\mu,\sigma}\circ\Phi_{i})+\eta\sum\nolimits_{i,j}\tilde{g}_j^{\kappa,\lambda}\left\{a_j\dot{\zeta}_{i}+(\boldsymbol{u}\cdot\nabla a_j)\zeta_{i}\right\}\Theta_j^{\mu,\sigma}\circ\Phi_{i}\,\mathrm{d}x\\
&=\int_{\mathbb{T}^d}\partial_t\theta_p+\boldsymbol{u}\cdot\nabla \theta_p\,\mathrm{d}x=\partial_t\int_{\mathbb{T}^d}\theta_p\,\mathrm{d}x=-\dot{\theta}_c.
\end{align*}
The construction of $\boldsymbol{R}_{\rm osc}$ is based on the following facts: recall $a_jb_j=-R_j\chi_j^2$, then we have
\begin{align*}
\theta_p\boldsymbol{w}_p&=-\sum\nolimits_{i,j}\tilde{g}_j^{\kappa,\lambda}\bar{g}_j^{\kappa,\lambda}R_j\chi_j^2\zeta_{i}^2[(\nabla \Phi_{i})^{-1}\boldsymbol{e}_j](\Theta_j^{\mu,\sigma}W_j^{\mu,\sigma})\circ\Phi_{i}\\
&\quad-\sum\nolimits_{j}\tilde{g}_j^{\kappa,\lambda}\bar{g}_j^{\kappa,\lambda}R_j\chi_j^2\sum_{i=1}^{D-1}\zeta_{i}\zeta_{i+1}[(\nabla \Phi_{i})^{-1}\boldsymbol{e}_j](W_j^{\mu,\sigma}\circ\Phi_{i})(\Theta_j^{\mu,\sigma}\circ\Phi_{i+1})\\
&\quad-\sum\nolimits_{j}\tilde{g}_j^{\kappa,\lambda}\bar{g}_j^{\kappa,\lambda}R_j\chi_j^2\sum_{i=1}^{D-1}\zeta_{i}\zeta_{i+1}[(\nabla \Phi_{i+1})^{-1}\boldsymbol{e}_j](W_j^{\mu,\sigma}\circ\Phi_{i+1})(\Theta_j^{\mu,\sigma}\circ\Phi_{i}),\\
\partial_t\theta_o(t,x)&=\sum\nolimits_{j}(\tilde{g}_j^{\kappa,\lambda}\bar{g}_j^{\kappa,\lambda}-1)\partial_j(\chi_j^2R_j)+\sum\nolimits_{j}\lambda^{-1} h_j^{\kappa,\lambda}(t)\partial_t\partial_j(\chi_j^2R_j)\\
&=\operatorname{div}\sum\nolimits_{j}[(\tilde{g}_j^{\kappa,\lambda}\bar{g}_j^{\kappa,\lambda}-1)\chi_j^2R_j+\lambda^{-1} h_j^{\kappa,\lambda}(t)\partial_t(\chi_j^2R_j)]\boldsymbol{e}_j,
\end{align*}
and
\begin{align*}
&\partial_t\theta_o+\operatorname{div}(\theta_p\boldsymbol{w}_p+\boldsymbol{R})\\
&=\operatorname{div}\Big\{\sum\nolimits_{j}[(\tilde{g}_j^{\kappa,\lambda}\bar{g}_j^{\kappa,\lambda}-1)\chi_j^2R_j+\lambda^{-1} h_j^{\kappa,\lambda}(t)\partial_t(\chi_j^2R_j)]\boldsymbol{e}_j\\
&\qquad\quad-\sum\nolimits_{i,j}\tilde{g}_j^{\kappa,\lambda}\bar{g}_j^{\kappa,\lambda}R_j\chi_j^2\zeta_{i}^2[(\nabla \Phi_{i})^{-1}\boldsymbol{e}_j](\Theta_j^{\mu,\sigma}W_j^{\mu,\sigma})\circ\Phi_{i}+\boldsymbol{R}_{\rm interact}+\boldsymbol{R}\Big\}\\
&=\operatorname{div}\Big\{\boldsymbol{R}_{\rm flow}+\boldsymbol{R}_{{\rm osc},t}-\sum\nolimits_{i,j}\tilde{g}_j^{\kappa,\lambda}\bar{g}_j^{\kappa,\lambda}R_j\chi_j^2\zeta_{i}^2[(\nabla \Phi_{i})^{-1}\boldsymbol{e}_j][(\Theta_j^{\mu,\sigma}W_j^{\mu,\sigma})\circ\Phi_{i}-1]+\boldsymbol{R}_{\rm interact}+\boldsymbol{R}_{\rm rem}\Big\}\\
&=\operatorname{div}\Big\{\boldsymbol{R}_{\rm flow}+\boldsymbol{R}_{{\rm osc},t}+\boldsymbol{R}_{{\rm osc},x}+\boldsymbol{R}_{\rm interact}+\boldsymbol{R}_{\rm rem}\Big\}.
\end{align*}

\subsection{Estimates on the perturbations}

\begin{lemma}\label{lem4-1}
\begin{align*}
&\|\theta_p\|_{L^s_tL^p_x}\lesssim \eta\|\boldsymbol{R}\|_{L^{1}_{t,x}}^{1/p}+\sigma^{-1/p}+\lambda^{-1/s},\\
&\|\boldsymbol{w}_p\|_{L^{s'}_tL^{p'}_x}\lesssim\eta^{-1}\|\boldsymbol{R}\|_{L^{1}_{t,x}}^{1/p'}+\sigma^{-1/p'}+\lambda^{-1/s'},\\
&\|\theta_c\|_{L^{\infty}_t}\lesssim\kappa^{\frac{1}{s}}\sigma^{-N}\mu^{\frac{d-1}{p}-\frac{d-1}{2}},\quad\|\theta_o\|_{L^{\infty}_{t,x}}\lesssim\lambda^{-1},\quad\|\boldsymbol{w}_c\|_{L^{s'}_tL^{p'}_x}\lesssim\sigma^{-1}.
\end{align*}
\end{lemma}
\begin{remark}\label{rem-infinity-case}
For $s'=\infty$, all the estimates in this section are still available except a slight modification for $\|\boldsymbol{w}_p\|_{L^{\infty}_tL^{p'}_x}$. Following the proof below we can deduce
\begin{equation*}
\|\boldsymbol{w}_p\|_{L^{\infty}_tL^{p'}_x}\lesssim\eta^{-1}\|\boldsymbol{R}\|_{L^{1}_{t,x}}^{1/p'}+\sigma^{-1/p'}.
\end{equation*}
Similarly, if $p'=\infty$, we can deduce
\begin{equation*}
\|\boldsymbol{w}_p\|_{L^{s'}_tL^{\infty}_x}\lesssim\eta^{-1}+\lambda^{-1/s'},
\end{equation*}
and if $s'=p'=\infty$, we can deduce
\begin{equation*}
\|\boldsymbol{w}_p\|_{L^{\infty}_{t,x}}\lesssim\eta^{-1},
\end{equation*}
\end{remark}
\begin{proof}
Recall the boundedness of $\|\tilde{g}_j^{\kappa,\lambda}\|_{L^s_t}$, $\|\Theta_j^{\mu,\sigma}\|_{L^p_x}$, $\|a_j\|_{C^1_{t,x}}$ by (\ref{eq-temporal-intermittency-estimate},\ref{eq-Mikado-density-estimate},\ref{eq-Lp-bound-ajbj}), we have
\begin{align*}
\|\theta_p\|_{L^s_tL^p_x}&\le\eta\sum\nolimits_{i,j}\|\tilde{g}_j^{\kappa,\lambda}a_j\zeta_{i}\Theta_j^{\mu,\sigma}\circ\Phi_{i}\|_{L^s_tL^p_x}\\
\text{\scriptsize \eqref{eq-Improved-Holder-classical} for $t$}&\lesssim\eta\sum\nolimits_{i,j}\|\tilde{g}_j^{\kappa,1}\|_{L^s_t}\Big(\|a_j\zeta_{i}\Theta_j^{\mu,\sigma}\circ\Phi_{i}\|_{L^s_tL^p_x}+\lambda^{-1/s}\|\Theta_j^{\mu,\sigma}\|_{L^p_x}\|a_j\|_{C^1_{t,x}}\Big)\\
\text{\scriptsize \eqref{eq-Improved-Holder} for $x$}&\lesssim\eta \Big(\sum\nolimits_{i,j}\|\Theta_j^{\mu,1}\|_{L^p_x}\big(\|a_j\zeta_{i}\|_{L^s_tL^p_x}+\sigma^{-1/p}\|a_j\|_{C^1_{t,x}}\big)+\lambda^{-1/s}\Big)\\
\text{\scriptsize (\ref{eq-Lp-bound-ajbj})}&\lesssim\eta \|\boldsymbol{R}\|_{L^{1}_{t,x}}^{1/p}+\sigma^{-1/p}+\lambda^{-1/s},
\end{align*}
where we have used the fact $\sum\nolimits_{i}\|a_j\zeta_{i}\|_{L^s_tL^p_x}\le 2\|a_j\|_{L^s_tL^p_x}$ according to the definition of $\zeta_{i}$ in Sec.~\ref{subsec-inverse-flow}. Notice the constants in front of $\sigma^{-1/p},\lambda^{-1/s}$ depend on $\eta$ and the definition of $\zeta_i$, however they will be absorbed as long as $p,s<\infty$ and $\sigma,\lambda$ large enough. Hence we can always assume that the constants hidden in $\lesssim$ are independent of $\eta$ and $\zeta_i$. The estimate on $\boldsymbol{w}_p$ can be obtained similarly.

For $\theta_c$, by \eqref{eq-Riemann-Lebesgue} and recall the boundedness of $\|a_j\|_{C^N_{t,x}}$, we have
\begin{align*}
\|\theta_c\|_{L^{\infty}_t}&\lesssim\sum\nolimits_{j}\|\tilde{g}_j^{\kappa,\lambda}\|_{L^{\infty}_t}\|a_j\|_{C^N_{t,x}}\sigma^{-N}\|\Theta_j^{\mu,1}\|_{L^2_x}\\
\text{\scriptsize (\ref{eq-Mikado-density-estimate},\ref{eq-temporal-intermittency-estimate})}&\lesssim\kappa^{\frac{1}{s}}\sigma^{-N}\mu^{\frac{d-1}{p}-\frac{d-1}{2}}.
\end{align*}

For $\theta_o$, we control $\|\nabla (\chi_j^2R_j)\|_{L^{\infty}_{t,x}}$ simply by a constant $C$, then by \eqref{eq-temporal-intermittency-estimate} we obtain
\begin{equation*}
\|\theta_o\|_{L^{\infty}_{t,x}}\le\lambda^{-1}\sum\nolimits_{j}\|h_j^{\kappa,\lambda}\|_{L^{\infty}_t}\|e_j\cdot\nabla (\chi_j^2R_j)\|_{L^{\infty}_{t,x}}\lesssim\lambda^{-1}.
\end{equation*}

For $\boldsymbol{w}_c$, recall the boundedness of $\|\bar{g}_j^{\kappa,\lambda}\|_{L^{s'}_t}$, $\|b_j\|_{C^2_{t,x}}$, $\|W_j^{\mu,1}\|_{L^{p'}_x}$ and by \eqref{eq-Wkp-bound-anti-with-inverse-map}, we have
\begin{align*}
\|\boldsymbol{w}_c\|_{L^{s'}_tL^{p'}_x}&\lesssim\sum\nolimits_{i,j}\|\bar{g}_j^{\kappa,\lambda}\|_{L^{s'}_t}\|b_j\|_{C^2_{t,x}}\|\nabla\Phi_{i}\|_{C^{0}(\mathbb{T}^d)}^{d-1}\sigma^{-1}\|W_j^{\mu,1}\|_{L^{p'}_x}\lesssim\sigma^{-1}.
\end{align*}
\end{proof}

\begin{lemma}\label{lem4-5}
\begin{align*}
\|\boldsymbol{w}_p\|_{L^{\tilde{s}}_tW^{1,\tilde{p}}_x}&\lesssim\kappa^{\frac{1}{s'}-\frac{1}{\tilde{s}}}\sigma\mu^{1+\frac{d-1}{p'}-\frac{d-1}{\tilde{p}}}\\
\|\boldsymbol{w}_c\|_{L^{\tilde{s}}_tW^{1,\tilde{p}}_x}&\lesssim\kappa^{\frac{1}{s'}-\frac{1}{\tilde{s}}}\mu^{1+\frac{d-1}{p'}-\frac{d-1}{\tilde{p}}}
\end{align*}
\end{lemma}
\begin{proof}
Recall the boundedness of $b_j$ in \eqref{eq-Ck-bound-ajbj}, we have
\begin{align*}
\|\boldsymbol{w}_p\|_{L^{\tilde{s}}_tW^{1,\tilde{p}}_x}&\le\eta^{-1}\sum\nolimits_{i,j}\|\bar{g}_j^{\kappa,\lambda}b_j\zeta_{i}[(\nabla \Phi_{i})^{-1}\boldsymbol{e}_j]W_j^{\mu,\sigma}\circ\Phi_{i}\|_{L^{\tilde{s}}_tW^{1,\tilde{p}}_x}\\
\text{\scriptsize \eqref{eq-3.2}}&\lesssim\eta^{-1}\sum\nolimits_{i,j}\|\bar{g}_j^{\kappa,\lambda}\|_{L^{\tilde{s}}_t}\|b_j\|_{C^{1}_{t,x}}\|W_j^{\mu,\sigma}\|_{W^{1,\tilde{p}}_x}\\
\text{\scriptsize (\ref{eq-temporal-intermittency-estimate},\ref{eq-Mikado-field-estimate})}&\lesssim\kappa^{\frac{1}{s'}-\frac{1}{\tilde{s}}}\sigma\mu^{1+\frac{d-1}{p'}-\frac{d-1}{\tilde{p}}},
\end{align*}
and
\begin{align*}
\|\boldsymbol{w}_c\|_{L^{\tilde{s}}_tW^{1,\tilde{p}}_x}&\lesssim\eta^{-1}\sum\nolimits_{i,j}\|\bar{g}_j^{\kappa,\lambda}\zeta_{i}\mathcal{R}\{\nabla b_j\cdot[(\nabla \Phi_{i})^{-1}\boldsymbol{e}_j]W_j^{\mu,\sigma}\circ\Phi_{i}\}\|_{L^{\tilde{s}}_tW^{1,\tilde{p}}_x}\\
\text{\scriptsize \eqref{eq-Wkp-bound-anti-with-inverse-map}}&\lesssim\eta^{-1}\sum\nolimits_{i,j}\|\bar{g}_j^{\kappa,\lambda}\|_{L^{\tilde{s}}_t}\|b_j\|_{C^3_{t,x}}\|W_j^{\mu,1}\|_{W^{1,\tilde{p}}_x}\\
\text{\scriptsize (\ref{eq-temporal-intermittency-estimate},\ref{eq-Mikado-field-estimate})}&\lesssim\kappa^{\frac{1}{s'}-\frac{1}{\tilde{s}}}\mu^{1+\frac{d-1}{p'}-\frac{d-1}{\tilde{p}}}.
\end{align*}

\end{proof}

\begin{lemma}\label{lem4-7}
For all $\phi\in C^{\infty}(\mathbb{T}^d)$, there hold
\begin{align*}
\Big\vert \int_{\mathbb{T}^d}\theta_p(t,x)\phi(x)\,dx\Big\vert&\lesssim\sigma^{-N}\mu^{\frac{d-1}{p}-\frac{d-1}{2}}\kappa^{\frac{1}{s}}\|\phi\|_{C^N_{x}},\\
\Big\vert \int_{\mathbb{T}^d}\theta_c(t,x)\phi(x)\,dx\Big\vert&\lesssim\sigma^{-N}\mu^{\frac{d-1}{p}-\frac{d-1}{2}}\kappa^{\frac{1}{s}}\|\phi\|_{L^{\infty}_{x}},\\
\Big\vert \int_{\mathbb{T}^d}\theta_o(t,x)\phi(x)\,dx\Big\vert&\lesssim\lambda^{-1}\|\phi\|_{L^{\infty}_{x}}.
\end{align*}
Moreover $\operatorname{supp}_t\theta\in I_r$.
\end{lemma}
\begin{proof}
By the definitions of $\chi_i$ and $a_j$ in Sec.~\ref{subsec-smooth-cutoff}, it is obvious that $\operatorname{supp}_t\theta\in I_r$. 
By \eqref{eq-Riemann-Lebesgue} and the boundedness of $\|a_j\|_{C^N_{t,x}}$ in \eqref{eq-Ck-bound-ajbj}, we have
\begin{align*}
\Big\vert \int_{\mathbb{T}^d}\theta_p(t,x)\phi(x)\,\mathrm{d}x\Big\vert &\lesssim\sum\nolimits_{i,j}\|\tilde{g}_j^{\kappa,\lambda}\|_{L^{\infty}}\|a_j\phi\|_{C^N_{t,x}}\sigma^{-N}\|\Theta_j^{\mu,1}\|_{L^2_x}\\
\text{\scriptsize (\ref{eq-temporal-intermittency-estimate},\ref{eq-Mikado-density-estimate})}&\lesssim\sigma^{-N}\mu^{\frac{d-1}{p}-\frac{d-1}{2}}\kappa^{\frac{1}{s}}\|\phi\|_{C^N_{x}}.
\end{align*}
The estimate on $\theta_c$ can be obtained similarly with $\phi=1$. Recall $\|\theta_o\|_{L^{\infty}_{t,x}}\lesssim\lambda^{-1}$ in Lemma~\ref{lem4-1}, we have
\begin{equation*}
\Big\vert \int_{\mathbb{T}^d}\theta_o(t,x)\phi(x)\,\mathrm{d}x\Big\vert \le\|\theta_o\|_{L^{\infty}_{t,x}}\|\phi\|_{L^{\infty}}\lesssim\lambda^{-1}\|\phi\|_{L^{\infty}_x}.
\end{equation*}
\end{proof}

\subsection{Estimates on the new defect field}

\begin{lemma}\label{lem-924-4.8}
\begin{align*}
\|\boldsymbol{R}_{\rm lin}\|_{L^{1}_{t,x}}&\lesssim\lambda^{-1}+\sigma^{-1}+\kappa^{-1/s}\mu^{-\frac{d-1}{p}},\\
\|\boldsymbol{R}_{\rm cor}\|_{L^{1}_{t,x}}&\lesssim\kappa^{\frac{1}{s}}\sigma^{-N}\mu^{\frac{d-1}{p}-\frac{d-1}{2}}+\lambda^{-1}+\sigma^{-1},\\
\|\boldsymbol{R}_{\rm trans}\|_{L^{1}_{t,x}}&\lesssim\lambda\kappa^{\frac{1}{s}}\sigma^{-1}\mu^{-1-\frac{d-1}{p'}},\\
\|\boldsymbol{R}_{{\rm osc},x}\|_{L^1_{t,x}}&\lesssim\sigma^{-1},\quad\|\boldsymbol{R}_{{\rm osc},t}\|_{L^{1}_{t,x}}\lesssim\lambda^{-1},\quad\|\boldsymbol{R}_{\rm rem}\|_{L^{1}_{t,x}}\le\frac{\delta}{4},\\
\|\boldsymbol{R}_{\rm flow}\|_{L^{1}_{t,x}}&\le\frac{\delta}{4},\quad\|\boldsymbol{R}_{\rm interact}\|_{L^1_{t,x}}\lesssim\lambda^{-1/2}.
\end{align*}
\end{lemma}

\begin{proof}

For $\boldsymbol{R}_{\rm lin}$, by the H\"older's inequality, we have
\begin{align*}
\|\boldsymbol{R}_{\rm lin}\|_{L^{1}_{t,x}}&\le\|\theta_o\|_{L^{\infty}_{t,x}} \|\boldsymbol{u}\|_{L^{1}_{t,x}}+\|\boldsymbol{w}\|_{L^{1}_{t,x}}\|\rho\|_{L^{\infty}_{t,x}}\\
&\lesssim\|\theta_o\|_{L^{\infty}_{t,x}}+\|\boldsymbol{w}_p\|_{L^{1}_{t,x}}+\|\boldsymbol{w}_c\|_{L^{s'}_tL^{p'}_x}.
\end{align*}
Notice
\begin{align*}
\|\boldsymbol{w}_p\|_{L^{1}_{t,x}}&\le\eta^{-1}\sum\nolimits_{i,j}\|\bar{g}_j^{\kappa,\lambda}b_j\zeta_{i}[(\nabla \Phi_{i})^{-1}\boldsymbol{e}_j]W_j^{\mu,\sigma}\circ\Phi_{i}\|_{L^{1}_{t,x}}\\&\lesssim\sum\nolimits_{j}\|\bar{g}_j^{\kappa,\lambda}\|_{L_t^1}\|W_j^{\mu,\sigma}\|_{L^{1}_{x}}\\
\text{\scriptsize (\ref{eq-temporal-intermittency-estimate},\ref{eq-Mikado-field-estimate})}
&\lesssim\kappa^{-1/s}\mu^{-\frac{d-1}{p}},
\end{align*}
and recall the estimates on $\theta_o,\boldsymbol{w}_c$ in Lemma~\ref{lem4-1}, we obtain the estimate on $\boldsymbol{R}_{\rm lin}$.

For $\boldsymbol{R}_{\rm cor}$, by the H\"older's inequality, we have
\begin{align*}
\|\boldsymbol{R}_{\rm cor}\|_{L^{1}_{t,x}}\le\|\theta\|_{L^{s}_{t}L^{p}_x} \|\boldsymbol{w}_c\|_{L^{s'}_{t}L^{p'}_x}+\|\theta_o+\theta_c\|_{L^{s}_{t}L^{p}_x}\|\boldsymbol{w}\|_{L^{s'}_{t}L^{p'}_x}.
\end{align*}
We control $\|\theta\|_{L^{s}_{t}L^{p}_x}$, $\|\boldsymbol{w}\|_{L^{s'}_{t}L^{p'}_x}$ simply by a constant $C$ and recall the estimates on $\boldsymbol{w}_c,\theta_o,\theta_c$ in Lemma~\ref{lem4-1}, we obtain the estimate on $\boldsymbol{R}_{\rm cor}$.

For $\boldsymbol{R}_{\rm trans,1}$, recall the bound of $\mathcal{R}$ in \eqref{eq-Wkp-bound-anti} we have
\begin{align*}
\|\boldsymbol{R}_{\rm trans,1}\|_{L^{1}_{t,x}}&\lesssim\sigma^{-1}\sum\nolimits_{i,j}\Big\|\partial_t(\tilde{g}_j^{\kappa,\lambda}a_j)\zeta_{i}(\nabla \Phi_{i})^{-1}\boldsymbol{\Omega}_j^{\mu,\sigma}\circ\Phi_{i}\Big\|_{L^{1}_{t,x}}\\
&\quad+\sigma^{-1}\sum\nolimits_{i,j}\Big\|\partial_t(\tilde{g}_j^{\kappa,\lambda}\nabla a_j)\cdot(\zeta_{i}(\nabla \Phi_{i})^{-1}\boldsymbol{\Omega}_j^{\mu,\sigma}\circ\Phi_{i})\Big\|_{L^{1}_{t,x}}\\
&\lesssim\sigma^{-1}\sum\nolimits_{j}\|\partial_t\tilde{g}_j^{\kappa,\lambda}\|_{L^{1}_{t}}\|\boldsymbol{\Omega}_j^{\mu,\sigma}\|_{L^{1}_{x}}\\
\text{\scriptsize (\ref{eq-temporal-intermittency-estimate},\ref{eq-Mikado-potential-estimate})}&\lesssim\lambda\kappa^{1/s}\sigma^{-1}\mu^{-1-\frac{d-1}{p'}}.
\end{align*}

For $\boldsymbol{R}_{\rm trans,2}$, recall the boundedness of $a_j$ in \eqref{eq-Ck-bound-ajbj}, we have
\begin{align*}
\|\boldsymbol{R}_{\rm trans,2}\|_{L^{1}_{t,x}}&\lesssim\eta\sum\nolimits_{i,j}\Big\|\mathcal{R}\Big\{\tilde{g}_j^{\kappa,\lambda}[a_j\dot{\zeta}_{i}+(\boldsymbol{u}\cdot\nabla a_j)\zeta_{i}]\Theta_j^{\mu,\sigma}\circ\Phi_{i}\Big\}\Big\|_{L^{1}_{t,x}}\\
\text{\scriptsize \eqref{eq-Wkp-bound-anti-with-inverse-map}}&\lesssim\sigma^{-1}\sum\nolimits_{j}\|\tilde{g}_j^{\kappa,\lambda}\|_{L^{1}_{t}}\|\Theta_j^{\mu,1}\|_{L^{1}_{x}}\\
\text{\scriptsize (\ref{eq-temporal-intermittency-estimate},\ref{eq-Mikado-density-estimate})}&\lesssim\kappa^{-1/s'}\sigma^{-1}\mu^{-\frac{d-1}{p'}}.
\end{align*}

For $\boldsymbol{R}_{{\rm osc},x}$, we have
\begin{align*}
\|\boldsymbol{R}_{{\rm osc},x}\|_{L^1_{t,x}}&\le\sum\nolimits_{i,j}\Big\|\tilde{g}_j^{\kappa,\lambda}\bar{g}_j^{\kappa,\lambda}\zeta_{i}^2\mathcal{R}\Big\{\nabla(R_j\chi_j^2)\cdot[(\nabla\Phi_{i})^{-1}\boldsymbol{e}_j][(\Theta_j^{\mu,\sigma}W_j^{\mu,\sigma})\circ\Phi_{i}-1]\Big\}\Big\|_{L^1_{t,x}}\\
\text{\scriptsize \eqref{eq-Wkp-bound-anti-with-inverse-map}}&\lesssim\sigma^{-1}\sum\nolimits_{j}\|\tilde{g}_j^{\kappa,\lambda}\bar{g}_j^{\kappa,\lambda}\|_{L^1_t}\|\Theta_j^{\mu,1}W_j^{\mu,1}-1\|_{L^1_x}\\
\text{\scriptsize (\ref{eq-temporal-intermittency-estimate},\ref{eq-Mikado-density-estimate},\ref{eq-Mikado-field-estimate})}&\lesssim\sigma^{-1}.
\end{align*}

For $\boldsymbol{R}_{{\rm osc},t}$, we control $\|\partial_t(\chi_j^2R_j)\|_{L^{\infty}_{t,x}}$ simply by a constant $C$, then we have
\begin{align*}
\|\boldsymbol{R}_{{\rm osc},t}\|_{L^{1}_{t,x}}&\le\lambda^{-1}\sum\nolimits_{j}\| h_j^{\kappa,\lambda}\partial_t(\chi_j^2R_j)\|_{L^1_{t,x}}\\
&\lesssim\lambda^{-1}\sum\nolimits_{j}\| h_j^{\kappa,\lambda}\|_{L^1_t}\|\partial_t(\chi_j^2R_j)\|_{L^{\infty}_{t,x}}\\
\text{\scriptsize \eqref{eq-temporal-intermittency-estimate}}&\lesssim\lambda^{-1}.
\end{align*}

For $\boldsymbol{R}_{\rm rem}$, according to the definition \eqref{eq-14} of $\chi_j$ and the choice \eqref{eq-15} of $r$, we have
\begin{align*}
\|\boldsymbol{R}_{\rm rem}\|_{L^{1}_{t,x}}&\le\sum\nolimits_{j}\|(1-\chi_j^2)R_j\boldsymbol{e}_j\|_{L^{1}_{t,x}}\\
&\le\sum\nolimits_{j}\int_{\vert R_j\vert \le\frac{\delta}{8d}}(1-\chi_j^2)\vert R_j\vert \,\mathrm{d}x\,\mathrm{d}t+\int_{t\in I^c_{r}}(1-\chi_j^2)\vert R_j\vert \,\mathrm{d}x\,\mathrm{d}t\\
&\le d\times(\frac{\delta}{8d}+2r\|\boldsymbol{R}\|_{L^{\infty}_{t,x}})=\frac{\delta}{4}.
\end{align*}

For $\boldsymbol{R}_{\rm flow}$, recall the boundedness of $\|\tilde{g}_j^{\kappa,\lambda}\bar{g}_j^{\kappa,\lambda}\|_{L_t^1}$, it is available to take $\nu$ small enough by \eqref{eq-3.5}, independent of $\mu,\sigma,\kappa,\lambda$, such that
\begin{align*}
\|\boldsymbol{R}_{\rm flow}\|_{L^{1}_{t,x}}&\lesssim\sum\nolimits_{i,j}\|\tilde{g}_j^{\kappa,\lambda}\bar{g}_j^{\kappa,\lambda}\chi_j^2\zeta_{i}^2[\operatorname{Id}-(\nabla\Phi_{i})^{-1}]R_j\|_{L^1_{t,x}}\\
&\lesssim\max_{i}\|\operatorname{Id}-(\nabla\Phi_{i})^{-1}\|_{L^{\infty}_{t,x}}\sum\nolimits_{j}\|\tilde{g}_j^{\kappa,\lambda}\bar{g}_j^{\kappa,\lambda}\|_{L_t^1}\\
&\le\frac{\delta}{4}.
\end{align*}

For $\boldsymbol{R}_{\rm interact}$, since $D=1/\nu$ is independent of $\mu,\sigma,\kappa,\lambda$, we have
\begin{align*}
\|\boldsymbol{R}_{\rm interact}\|_{L^{1}_{t,x}}&\lesssim\sum\nolimits_{j}\sum_{i=1}^{D-1}\|\tilde{g}_j^{\kappa,\lambda}\bar{g}_j^{\kappa,\lambda}\zeta_{i}\zeta_{i+1}\|_{L^{1}_{t}}\\
\text{\scriptsize \eqref{eq-Improved-Holder-classical}}&\lesssim\sum_{i=1}^{D-1}(\|\zeta_{i}\zeta_{i+1}\|_{L^{1}_{t}}+\lambda^{-1}\|\zeta_{i}\zeta_{i+1}\|_{C^{1}_{t}})\sum\nolimits_{j}\|\tilde{g}_j^{\kappa,1}\bar{g}_j^{\kappa,1}\|_{L^{1}_{t}}\\
\text{\scriptsize \eqref{eq-3.6}}&\lesssim\lambda^{-1/2}.
\end{align*}

\end{proof}

\subsection{Conclusion}

There are four controllable parameters $\mu,\sigma,\kappa,\lambda$. Let the spatial intermittency $\mu$ large enough, spatial oscillation $\sigma\in\mathbb{N}$, $\sigma\simeq\mu^{\alpha}$, and temporal intermittency $\kappa=\mu^{\beta}$, temporal oscillation $\lambda\in\mathbb{N}$, $\lambda\simeq\mu^{\gamma}$ for some $\alpha,\beta,\gamma>0$.

To prove Proposition~\ref{prn3.1}, by Lemma~\ref{lem4-1}-\ref{lem-924-4.8}, the assumption $s,p<\infty$, Remark~\ref{rem-infinity-case}, $\alpha,\beta,\gamma>0$ and appropriate choice of $N$ large enough depending only on $\alpha,\beta,\gamma$, it is sufficient to obtain
\begin{align*}
\kappa^{\frac{1}{s'}-\frac{1}{\tilde{s}}}\sigma\mu^{1+\frac{d-1}{p'}-\frac{d-1}{\tilde{p}}}&\ll 1,\\
\lambda\kappa^{\frac{1}{s}}\sigma^{-1}\mu^{-1-\frac{d-1}{p'}}&\ll 1,
\end{align*}
which is equivalent to 
\begin{align*}
\alpha+1+\frac{d-1}{p'}-\frac{d-1}{\tilde{p}}+\frac{\beta}{s'}-\frac{\beta}{\tilde{s}}&<0,\\
-\alpha-1-\frac{d-1}{p'}+\gamma+\frac{\beta}{s}&<0.
\end{align*}
Since $\tilde{s}<s'$, the inequalities are equivalent to 
\begin{equation*}
\left(\alpha+1+\frac{d-1}{p'}-\frac{d-1}{\tilde{p}}\right)\left(\frac{1}{\tilde{s}}-\frac{1}{s'}\right)^{-1}<\beta<\left(\alpha+1+\frac{d-1}{p'}-\gamma\right)s.
\end{equation*}
Take $\gamma>0$ small enough, then $\beta>0$ exists if 
\begin{equation*}
\left(\alpha+1+\frac{d-1}{p'}-\frac{d-1}{\tilde{p}}\right)\left(\frac{1}{\tilde{s}}-\frac{1}{s'}\right)^{-1}<\left(\alpha+1+\frac{d-1}{p'}\right)s,
\end{equation*}
which is equivalent to
\begin{equation*}
\alpha<\frac{(d-1)\tilde{s}'}{s\tilde{p}}-1-\frac{d-1}{p'}.
\end{equation*}
Now to make sure that $\alpha>0$ exists, it is sufficient to require
\begin{equation*}
\frac{(d-1)\tilde{s}'}{s\tilde{p}}-1-\frac{d-1}{p'}>0,
\end{equation*}
which is exactly the condition \eqref{eq-assum}.

\section{Proof of Theorem~\ref{thm1.2}} \label{sec-proof-of-theorem-2}

In order to prove Theorem~\ref{thm1.2}, we only need to add minor adjustments and the estimates of $\|\theta\|_{L_t^{\bar{s}}C_x^{\bar{m}}}$, $\|\boldsymbol{R}_{\rm diffusion}\|_{L_{t,x}^1}$ to the proof presented in Sec.~\ref{sec-main-proposition} and \ref{sec-proof-of-main-proposition}, where $\boldsymbol{R}_{\rm diffusion}$ satisfies
\begin{align*}
\operatorname{div}\boldsymbol{R}_{\rm diffusion}:=\mathcal{L}_k\theta.
\end{align*}
Note
\begin{align*}
\theta&=\theta_p+\theta_c+\theta_o\\
&=\sum\nolimits_{i,j}\eta \tilde{g}_j^{\kappa,\lambda}(t)a_j(t,x)\zeta_i(t)\Theta_j^{\mu,\sigma}\circ\Phi_{i}-\fint_{\mathbb{T}^d}\theta_p(t,x)\,\mathrm{d}x+\sum\nolimits_{j}\lambda^{-1} h_j^{\kappa,\lambda}(t)\partial_j(\chi_j^2R_j)\\
&=\sum\nolimits_{i,j}\sigma^{-1}\eta\tilde{g}_j^{\kappa,\lambda}\zeta_i\left\{\operatorname{div}[a_j(\nabla \Phi_{i})^{-1}\boldsymbol{\Omega}_j^{\mu,\sigma}\circ\Phi_{i}]-\nabla_xa_j\cdot[(\nabla \Phi_{i})^{-1}\boldsymbol{\Omega}_j^{\mu,\sigma}\circ\Phi_{i}]\right\}\\
&\quad-\fint_{\mathbb{T}^d}\theta_p(t,x)\,\mathrm{d}x+\operatorname{div}\sum\nolimits_{j}\lambda^{-1} h_j^{\kappa,\lambda}\chi_j^2R_j\boldsymbol{e}_j.
\end{align*}
Hence $\boldsymbol{R}_{\rm diffusion}$ can be constructed as
\begin{align*}
\boldsymbol{R}_{\rm diffusion}&=\sum\nolimits_{i,j}\sigma^{-1}\eta\tilde{g}_j^{\kappa,\lambda}\zeta_i\mathcal{L}_k\left[a_j(\nabla \Phi_{i})^{-1}\boldsymbol{\Omega}_j^{\mu,\sigma}\circ\Phi_{i}\right]+\sum\nolimits_{j}\lambda^{-1} h_j^{\kappa,\lambda}\mathcal{L}_k(\chi_j^2R_j)\boldsymbol{e}_j\\
&\quad-\sum\nolimits_{i,j}\sigma^{-1}\eta\tilde{g}_j^{\kappa,\lambda}\zeta_i\mathcal{R}\mathcal{L}_k\{\nabla_xa_j\cdot[(\nabla \Phi_{i})^{-1}\boldsymbol{\Omega}_j^{\mu,\sigma}\circ\Phi_{i}]\}
\end{align*}
Then by the analysis analogous to Sec.~\ref{sec-proof-of-main-proposition}, we have
\begin{align*}
\|\theta\|_{L_t^{\bar{s}}C_x^{\bar{m}}}&\lesssim\kappa^{\frac{1}{s}-\frac{1}{\bar{s}}}\sigma^{\bar{m}}\mu^{\bar{m}+\frac{d-1}{p}},\\
\|\boldsymbol{R}_{\rm diffusion}\|_{L_{t,x}^1}&\lesssim\kappa^{-1/s'}\sigma^{k-1}\mu^{k-1-\frac{d-1}{p'}},
\end{align*}
and it is sufficient to prove the existence of $\alpha,\beta,\gamma>0$ under the assumptions in the theorem, such that
\begin{align*}
(\alpha+1)\bar{m}+\frac{d-1}{p}+\frac{\beta}{s}-\frac{\beta}{\bar{s}}&<0,\\
\alpha+1+\frac{d-1}{p'}-\frac{d-1}{\tilde{p}}+\frac{\beta}{s'}-\frac{\beta}{\tilde{s}}&<0,\\
-\alpha-1-\frac{d-1}{p'}+\gamma+\frac{\beta}{s}&<0,\\
(\alpha+1)(k-1)-\frac{d-1}{p'}-\frac{\beta}{s'}&<0.
\end{align*}
Since $\bar{s}<s$, $\tilde{s}<s'$, the inequalities are equivalent to 
\begin{align*}
&\max\{\beta_1,\beta_2,\beta_3\}<\beta<\beta_4,\\
&\beta_1=\left((\alpha+1)\bar{m}+\frac{d-1}{p}\right)\left(\frac{1}{\bar{s}}-\frac{1}{s}\right)^{-1},\\
&\beta_2=\left(\alpha+1+\frac{d-1}{p'}-\frac{d-1}{\tilde{p}}\right)\left(\frac{1}{\tilde{s}}-\frac{1}{s'}\right)^{-1},\\
&\beta_3=\left((\alpha+1)(k-1)-\frac{d-1}{p'}\right)s',\\
&\beta_4=\left(\alpha+1+\frac{d-1}{p'}-\gamma\right)s.
\end{align*}
Take $\gamma>0$ small enough, then $\beta>0$ exists if $\beta_4^0>\max\{\beta_1,\beta_2,\beta_3\}$ where
\begin{equation*}
\beta_4^0=\left(\alpha+1+\frac{d-1}{p'}\right)s.
\end{equation*}
Note by the assumptions $\bar{m}<\frac{s}{\bar{s}}-1$,  $k<\frac{s}{s'}+1$, $\beta_4^0>\beta_1$ is equivalent to
\begin{equation*}
\alpha>(d-1)\left(\frac{s}{\bar{s}p'}-1\right)\left(\bar{m}+1-\frac{s}{\bar{s}}\right)^{-1}-1,
\end{equation*}
$\beta_4^0>\beta_2$ is equivalent to
\begin{equation*}
\alpha<\frac{(d-1)\tilde{s}'}{s\tilde{p}}-1-\frac{d-1}{p'},
\end{equation*}
$\beta_4^0>\beta_3$ is equivalent to
\begin{equation*}
\alpha>\frac{d-1}{p'}\left(\frac{s}{s'}-1\right)\left(k-1-\frac{s}{s'}\right)^{-1}-1.
\end{equation*} 

\noindent{\bf The case $\tilde{s}>1$.} To make sure that $\alpha>0$ exists, it is sufficient to require $\alpha_1>0$ and $\alpha_1>\max\{\alpha_2,\alpha_3\}$, where
\begin{align*}
&\alpha_1=\frac{(d-1)\tilde{s}'}{s\tilde{p}}-1-\frac{d-1}{p'},\\
&\alpha_2=(d-1)\left(\frac{s}{\bar{s}p'}-1\right)\left(\bar{m}+1-\frac{s}{\bar{s}}\right)^{-1}-1,\\
&\alpha_3=\frac{d-1}{p'}\left(\frac{s}{s'}-1\right)\left(k-1-\frac{s}{s'}\right)^{-1}-1.
\end{align*}
It is equivalent to require
\begin{align*}
&\frac{\tilde{s}'}{s\tilde{p}}-\frac{1}{p'}>\frac{1}{d-1},\\
&\frac{\tilde{s}'}{s\tilde{p}}-\frac{1}{p'}>\left(\frac{s}{\bar{s}p'}-1\right)\left(\bar{m}+1-\frac{s}{\bar{s}}\right)^{-1},\\
&\frac{\tilde{s}'}{s\tilde{p}}-\frac{1}{p'}>\frac{1}{p'}\left(\frac{s}{s'}-1\right)\left(k-1-\frac{s}{s'}\right)^{-1}.
\end{align*}
We reduce these inequalities to the first inequality by adding restrictions on $\bar{m}$ and $k$, i.e., let
\begin{align*}
\frac{1}{d-1}&\ge\left(\frac{s}{\bar{s}p'}-1\right)\left(\bar{m}+1-\frac{s}{\bar{s}}\right)^{-1},\\
\frac{1}{d-1}&\ge\frac{1}{p'}\left(\frac{s}{s'}-1\right)\left(k-1-\frac{s}{s'}\right)^{-1},
\end{align*}
which is equivalent to
\begin{equation}
\bar{m}<\frac{s}{\bar{s}}-1-(d-1)\left(\frac{s}{\bar{s}p'}-1\right)^{-},\quad k<\frac{s}{s'}+1-\frac{d-1}{p'}\left(\frac{s}{s'}-1\right)^{-}.
\end{equation}

\noindent{\bf The case $\tilde{s}=1$.} In this case, $\beta_4^0>\beta_2$ is always true, hence $\alpha$ can be taken large enough such that $\beta_4^0>\max\{\beta_1,\beta_3\}$ as long as $\bar{m}<\frac{s}{\bar{s}}-1$,  $k<\frac{s}{s'}+1$.

\section*{Acknowledgments}
The author thanks the anonymous reviewers for their valuable suggestions. 

%
%

\bibliographystyle{abbrv}
\bibliography{ref}
\end{document}